\documentclass[12 pt,a4paper,twoside,reqno]{amsart}
\usepackage{amsfonts,amssymb,amscd,amsmath,enumerate,verbatim,calc}

\usepackage[applemac]{inputenc}
\usepackage[pdftex]{graphicx}
 
 \usepackage
{hyperref}

\usepackage{xcolor}
	\hypersetup{
 colorlinks,
 linkcolor={red!50!black},
 citecolor={blue!50!black},
 urlcolor={blue!80!black}
}

\usepackage{xy}
\usepackage{xypic} 
 
\usepackage{multicol}  
\usepackage{amsthm} 
\usepackage{mathrsfs} 
\usepackage{mathtools}

\def\C{{\mathbb C}} 
\def\N{{\mathbb N}}

\def\Z{{\mathbb Z}}

\def\calE{{\mathcal E}}

\def\ftilde{\tilde{f}}

\def\rmd{{\mathrm d}}
\def\rme{{\mathrm e}}
\def\rmi{{\mathrm i}}
\renewcommand{\and}{\; \hbox{ and }\; }
\renewcommand{\for}{\; \hbox{ for }\; }
 
\newtheorem*{theorem*}{Theorem}
\newtheorem{theorem}{Theorem}
\newtheorem{corollary}{Corollary}
\newtheorem{proposition}{Proposition}
\newtheorem{lemma}{Lemma}

\def\atop#1#2{
\genfrac{}{}{0pt} {} 
{#1} 
{#2}}

\begin{document} 
 
 \null
 \vskip -3 true cm

 \hfill{\small 
Update: {\em \today
}}

{\small 

\noindent
 \href{https://www.indianmathsociety.org.in/ms.htm}{The Mathematics Student}, February 2022.
\\
Proceedings of the 87th Annual Conference of the 
 \\
 Indian Mathematical Society, December 2021. 
 }
 
 \bigskip
\begin{center}
{
\Large
\bf 
Lidstone interpolation 

I. One variable
\\

\bigskip
\large
 Michel Waldschmidt
 
 \hskip .01 true cm
 \vtop{\hsize 1cm \kern 1mm
 \hrule width 16mm \kern 1mm}
 }

 \sl
 Sorbonne Université and Université de Paris, 
 \\
 CNRS, IMJ-PRG, F-75005 Paris, France

 \tt
 \href{mailto:michel.waldschmidt@imj-prg.fr}{michel.waldschmidt@imj-prg.fr}
 
 \url{http://www.imj-prg.fr/~michel.waldschmidt}

\end{center}

\tableofcontents

\section*{Abstract} 
 
According to Lidstone interpolation theory, an entire function of exponential type $<\pi$ is determined by it derivatives of even order at $0$ and $1$. This theory can be generalized to several variables. Here we survey the theory for a single variable. Complete proofs are given. This first paper of a trilogy is devoted to Univariate Lidstone interpolation; Bivariate and Multivariate Lidstone interpolation will be the topic of two forthcoming papers. 
 
\bigskip
\subsection*{Keywords} 
Lidstone polynomials, exponential type, analytic functions of one variable.

\bigskip
\noindent
{\bf AMS Mathematics Subject Classification 2020}:

Primary:  
41A58  

Secundary: 
30D15 

\bigskip
\noindent
{\bf Acknowledgments}. This is an expanded version of the first part of a plenary talk given at  the 87th Annual Conference of the 
 Indian Mathematical Society in December 2021. The second part of the talk was devoted to several variables. This paper and the two forthcoming ones on two and several variables were completed at  The Institute of Mathematical Sciences (IMSc) Chennai, India, where the author gave a course in December 2021 under the Indo-French Program for Mathematics. The author is thankful to IMSc for its hospitality and the Laboratoire International  Franco--Indien for its support.

 \section{Introduction} 
 
 In 1930, in a seminal paper \cite{zbMATH02567100}, G.~J. Lidstone introduced a basis $\Lambda_k(z)$ ($k\ge 0$) of the space $\C[z]$ of polynomials in a single variable, which has the property that any polynomial $f\in\C[z]$ has a finite expansion 
 $$
 f(z)=\sum_{k\ge 0} f^{(2k)}(0)\Lambda_k(1-z)+\sum_{k\ge 0} f^{(2k)}(1)\Lambda_k(z),
$$
where 
$$
f^{(2k)} =\left(\frac \rmd {\rmd z}\right)^{2k} f.
$$
Two years later, H.~Poritsky \cite{Poritsky}  J.~M.~Whittaker \cite{Whittaker} extended these expansions to entire functions of exponential type $<\pi$. In 1936, I.J.~Schoenberg \cite{Schoenberg} proved that the only entire functions of finite exponential type which vanish at the two points $0$ and $1$ together with all their derivatives of even order are the linear combinations with constant coefficients of the functions $\sin(k\pi z)$, with $k\in\N$. This result follows from an expansion formula for such functions which was obtained by R.C. Buck in 1955 \cite{Buck}. 

We recall the basic facts concerning Lidstone expansion in a single variable. In this section $z$ and $\zeta$ are in $\C$. It will be convenient to use notations which can be generalized to several variables. The classical Lidstone polynomials \cite[\S 6 p.~18]{zbMATH02567100} 
 $\Lambda_k(z)$ ($k\ge 0$)  are denoted here $\Lambda_{2k,1}(z)$,  
 while the polynomials $\Lambda_k(1-z)$ are written here $\Lambda_{2k,0}(z)$. The successive derivatives of a function $f$ of a single variable are denoted $f'$, $f''$, \dots, $f^{(k)}$. 

 The Lidstone polynomials are introduced in Theorem \ref{Theorem:Lidstone}. They are the solution of a system of differential equations (Lemma \ref{Lemma:equadiffunevariable}).
The unicity of the expansion for an entire function of exponential type $<\pi$ (Theorem \ref{TheoremPoristkyUnicite-n=1}) is easy to prove, the existence (Theorem \ref{Theorem:Poritsky}) needs more work - both results are due to H.~Poritsky and J.~M.~Whittaker. Next we prove integral formulae for the Lidstone polynomials (Propositions \ref{Prop:IntegralFormulaLidstone} and \ref{Prop:IntegralFormula2}), and we give proofs of the results of Buck (Proposition \ref{Prop:Buck}) and Schoenberg (Corollary \ref{Corollary:Schoenberg}).

 This paper is self contained, full proofs are given. It is an introduction to two forthcoming papers,  
\cite{BivariateLidstone}
where we extend the theory to two variables
and 
\cite{MultivariateLidstone}
where we extend the theory to an arbitrary number of variables.

 \section{Definition of the univariate Lidstone polynomials} \label{SS:Introduction1}
Let us recall the definitions of the \emph{order} of an entire function $f$:
$$
\varrho(f)=\limsup_{r\to\infty} \frac{\log\log |f|_r}{\log r}
\;
\text{ where } 
\;
|f|_r=\sup_{|z|=r}|f(z)|
$$
and of the 
\emph{exponential type} of $f$: 
$$
\tau(f)= \limsup_{r\to\infty} \frac{ \log |f|_r}{r}\cdotp
$$
If the exponential type of $f$ is finite, then $f$ has order $\le 1$. If $f$ has order $<1$, then the exponential type is $0$. 
Using Cauchy's estimate for the coefficients of the Taylor series together with Stirling's formula for $n!$, one deduces
\cite[Lemma 1]{Whittaker}  that if $f$ has exponential type $\tau(f)$, then for all $z_0\in\C$, 
$$
\limsup_{n\to\infty} |f^{(n)}(z_0)|^{1/n}= \tau(f).
$$ 
For $\zeta\in\C\setminus\{0\}$, the function $\rme^{\zeta z}$ has order $1$ and exponential type $|\zeta|$.

 \bigskip
 We denote by  $2\N$ the set of even nonnegative integers.
 The starting point of the theory of Lidstone interpolation is the following.

\begin{lemma} \label{Lemma:UneVariable}
Let $f$ be a polynomial satisfying 
\begin{equation}\label{equation:lemme1}
 f^{(t)}(0)= f^{(t)}(1)=0 \hbox{ for all $t \in 2\N$}.
\end{equation}
Then $f=0$.
\end{lemma}

We give three proofs of this lemma, the arguments are slightly different and will be used again. 
 
\begin{proof}[First proof]
By induction on the total degree of the polynomial $f$.

If $f$ has degree $\le 1$, say $f(z)=a_0z+a_1$, the conditions $f(0)=f(1)=0$ imply $a_0=a_1=0$, hence $f=0$.

If $f$ has degree $\le d$ with $d\ge 2$ and satisfies the hypotheses, then $f''$ also satisfies the hypotheses and has degree $<d$, hence by induction $f''=0$ and therefore $f$ has degree $\le 1$. 
 
Lemma \ref{Lemma:UneVariable} follows. 
\end{proof}

\begin{proof}[Second proof]
Let $f$ be a polynomial satisfying \eqref{equation:lemme1}.
The assumption $f^{(t)}(0)=0$ for all $t \in 2\N$ means that $f$ is an odd function: $f(-z)=-f(z)$. 
The assumption $ f^{(t)}(1)=0$ for all $t\in 2\N$  means that $f(1-z)$ is an odd function: $f(1-z)=-f(1+z)$. 
We deduce 
$$
f(z+2)=f(1+z+1)=-f(1-z-1)=-f(-z)=f(z),
$$
hence the polynomial $f$ is periodic, and therefore is a constant. 
 
Since $f(0)=0$, we conclude $f=0$. 
\end{proof}

\begin{proof}[Third proof]
 Assume \eqref{equation:lemme1}. Write 
 $$
 f(z)=a_1z+a_3 z^3+a_5z^5+a_7z^7++\cdots+a_{2m+1}z^{2m+1}+\cdots
 $$
(finite sum). 
 We have $f(1)=f''(1)=f^{\mathrm{(iv})}(1)=\cdots=0$:
 $$ 
 \begin{matrix}
 a_1&+a_3 &+a_5 \hfill&+a_7\hfill
 &+\cdots&+a_{2n+1}\hfill&+\cdots=0
 \\ 
 	& \;\; 6a_3&+20a_5&+42a_7 \hfill&+\cdots&+2m(2m+1)a_{2m+1}&+\cdots=0
	\\ 
 	& &\; \;120a_5&+840a_7 &+\cdots&+\frac{(2m+1)!}{(2m-3)!}a_{2m+1}\hfill&+\cdots=0
	\\
	& & & & & \ddots &\vdots 
 \end{matrix}
 $$
The matrix of this system is triangular with maximal rank. We conclude $a_1=a_3=a_5=\cdots=0$.
\end{proof}
 
 The fact that this matrix has maximal rank means that a polynomial $f$ is uniquely determined by the numbers 
$$
 f^{(t)}(0) \and f^{(t)}(1) \for t\in 2\N.
$$

Let $T\ge 0$ be even. The space $\C[z]_{\le T+1}$ of polynomials of degree $\le T+1$ has dimension $T+2$. All elements $f\in \C[z]_{\le T+1}$ satisfy $f^{(k)}=0$ for $k\ge T+2$. Lemma \ref{Lemma:UneVariable} shows that the linear map 
 $$
 \begin{matrix}
 \C[z]_{\le T+1}& \longrightarrow& \C^{T+2}
 \\
 f&\longmapsto & \bigl( f^{(t)}(0),\; f^{(t)}(1)\bigr)_{0\le t\le T, \; t \in 2\N}
 \end{matrix}
 $$
 is injective. Hence it is an isomorphism.

Given numbers $a_t$ and $b_t$, ($t\in 2\N$), where all but finitely many of them are $0$, there is a unique polynomial $f$ such that 
$$
 f^{(t)}(0)=a_t \and f^{(t)}(1) =b_t\hbox{ for all $t\in 2\N$}.
$$
 In particular, for each  $t\in 2\N$, there is a unique polynomial $\Lambda_{t,0}$ which satisfies 
$$
\Lambda^{(\tau)}_{t,0}(0)= \delta_{t,\tau} \and \Lambda^{(\tau)}_{t,0}(1)=0 \for \tau \in 2\N
$$
(Kronecker symbol), and there is a unique polynomial $\Lambda_{t,1}$ which satisfies 
$$
\Lambda^{(\tau)}_{t,1}(0)=0 \and \Lambda^{(\tau)}_{t,1}(1)=\delta_{t,\tau} \for \tau \in 2\N.
$$
Therefore:

 \begin{theorem}
[G. J.~Lidstone (1930)]\label{Theorem:Lidstone}
There exist two sequences of polynomials, $\bigl(\Lambda_{t,0}(z)\bigr)_{t\in 2\N}$,  $\bigl(\Lambda_{t,1}(z)\bigr)_{t \in 2\N}$, such that any polynomial $f$ can be written as a finite sum 
\begin{equation}\label{Equation:LidstoneExpansionUneVariable}
f(z)=\sum_{t\in 2\N}
 f^{(t)}(0) \Lambda_{t,0}(z)+
\sum_{t\in 2\N} f^{(t)}(1) \Lambda_{t,1}(z).
\end{equation}
\end{theorem}

The involution  $z\mapsto 1-z$:
$$
0\mapsto 1, \quad 1\mapsto 0, \quad 1-z\mapsto z
$$
 shows that $\Lambda_{t,0}(z)=\Lambda_{t,1}(1-z)$.

 At this point, we can make an analogy with Taylor series, where the polynomials $z^m/m!$
satisfy
$$
\frac{\rmd^k}{\rmd z^k}
\left(
\frac{z^m}{m!}
\right)_{z=0}
=\delta_{mk} \for m\ge 0
\and k\ge 0.
$$ 
Given a sequence $(a_m)_{m\ge 0}$ of complex numbers, the unique analytic solution (if it exists) $f$ of the interpolation problem 
$$
f^{(m)}(0)=a_m\hbox{ for all $m\ge 0$}
$$
is given by the Taylor expansion
 $$
f(z)=\sum_{m\ge 0} a_m\frac{z^m}{m!}\cdotp
$$
Lidstone expansion  replaces the single point $0$ and the sequence of all derivatives with two points $0$ and $1$ and only the derivatives of even order at these two points. 

The first Lidstone polynomial is $ \Lambda_{0,1}(z)=z$:
 $$
 \Lambda_{0,1}(0)=0, \quad  \Lambda_{0,1}(1)=1, \quad  \Lambda_{0,1}^{(t)}(0)=\Lambda_{0,1}^{(t)}(1)=0 \for t\in 2\N, \;  t\ge 2.
 $$
 The next lemma provides an  inductive way for finding all of them. 
 
  \section{Differential equation}\label{SS:equadiff1}
  
\begin{lemma}\label{Lemma:equadiffunevariable}
The sequence of Lidstone polynomials $\bigl( \Lambda_{t,1}\bigr)_{t\in 2\N}$ is determined by $ \Lambda_{0,1}(z)=z$ and 
$$
\Lambda''_{t,1}=\Lambda_{t-2,1} \for t\ge 2 \text{ even},
$$
with the initial conditions $ \Lambda_{t,1}(0)=\Lambda_{t,1}(1)=0$ for $t\in 2\N$, $t\ge 2$. 
More precisely, let $\bigl( L_t \bigr)_{t \in 2\N}$  be a sequence of polynomials satisfying 
$ L_0(z)=z$ and 
$$
L''_t=L_{t-2} \for  t\in 2\N, \;  t\ge 2,
$$
with the initial conditions $ L_t(0)=L_t(1)=0$ for $t\in 2\N$, $t\ge 2$; then $L_t=\Lambda_{t,1}$ for all $t \in 2\N$. 
\end{lemma}

Notice that the assumption $ L_0(z)=z$ cannot be omitted: given any polynomial $A$, there is a unique sequence $\bigl( L_t \bigr)_{t \in 2\N}$  satisfying all other assumptions but with $L_0=A$. 
 
   \begin{proof}
  That the sequence $\bigl( \Lambda_{t,1}\bigr)_{t\in 2\N}$ satisfies these conditions is plain. We now prove the unicity. Let $\bigl(L_t\bigr)_{t  \in 2\N}$,  be a sequence of polynomials satisfying the conditions of Lemma \ref{Lemma:equadiffunevariable}. By assumption  $L_0(z)=z$. By induction, assume that  for some $t\ge 2$ we know that $L_{t-2}=\Lambda_{t-2,1}$. Then the difference $g=L_t-\Lambda_{t,1}$ satisfies $g''=0$, hence $g$ has degree $\le 1$. The assumptions $ L_t(0)=L_t(1)=0$ for $t\in 2\N$, $t\ge 2$ imply $g=0$. 
  \end{proof}

For $t \in 2\N$, the polynomial $\Lambda_{t,1}$ is odd, it has degree $t+1$ and leading term $\frac{1}{(t+1)!}z^{t+1}$.
For instance
$$ 
\Lambda_{2,1}(z)=\frac{1}{6} (z^3-z)=\frac{1}{6}z (z-1)(z+1), 
$$
$$
 \Lambda_{2,0}(z)=\Lambda_{2,1}(1-z)=-\frac{z^3}{6}+\frac{z^2} 2 -\frac z 3=-\frac{1}{6} z(z-1)(z-2),
$$ 
$$
 \Lambda_{4,1}(z)=\frac{1}{120} z^5-\frac{1}{36} z^3+\frac{7}{360} z= \frac{1}{360}z(z^2-1)(3z^2-7),
 $$ 
 and
\begin{equation}\label{eq:Lambda40}
\begin{aligned}
 \Lambda_{4,0}(z)=\Lambda_{4,1}(1-z)
 &=-\frac{1}{120} z^5+\frac{1}{24} z^4-\frac{1}{18} z^3+\frac{1}{45} z
 \\ 
 &=- \frac{1}{360}z(z-1)(z-2)(3z^2-6z-4).
\end{aligned}
\end{equation}

  \section{Recurrence formula}\label{SS:recurrence}

 For $t \in 2\N$, the polynomial $f_t(z)= z^{t+1}$ satisfies 
$$
f_t^{(\tau)}(0)=0  \for \tau \in 2\N, \quad 
f_t^{(\tau)}(1)=
\begin{cases} 
\frac {(t+1)!}{(t-\tau+1)!}&\hbox{ {for} $0\le \tau\le t$, $\tau\in 2\N$}
\\
0&\hbox{ {for} $\tau\ge t+2$, $\tau\in2\N$.}
\end{cases}
$$
From Theorem \ref{Theorem:Lidstone} one deduces, for $t \in 2\N$,
$$
 z^{t+1} =\sum_{\atop{0\le \tau\le t}{\tau \in 2\N}} \frac {(t+1)!} {(t-\tau+1)!} \Lambda_{\tau,1}(z), 
$$
which yields the recurrence formula
\begin{equation}\label{inductiveformulaunevariable}
 \Lambda_{t,1}(z) =\frac 1 {(t+1)!} z^{t+1} -\sum_{\atop{0\le \tau\le t-2}{\tau\in 2\N }} \frac {1} {(t-\tau+1)!} \Lambda_{\tau,1}(z).
\end{equation} 
 Another consequence of  Theorem \ref{Theorem:Lidstone} is
\begin{equation}\label{inductiveformulaunevariableLambdat0}
\frac{z^{t}}{t!}= \Lambda_{t,0}(z)+
\sum_{\atop{0\le \tau\le t}{\tau \in 2\N }} 
\frac 1 {(t-\tau)!} \Lambda_{\tau,1}(z) 
\end{equation} 
for $t\in 2\N$.

  \section{Unicity for entire functions}\label{SS:Unicity}
 According to Theorem \ref{Theorem:Lidstone}, a polynomial is determined by the values of its derivatives of even order at the two points $0$ and $1$. H.~Poritsky \cite{Poritsky} and J.~M.~Whittaker \cite{Whittaker} proved that he same is true more generally for  an entire function of  exponential type $<\pi$:
 
\begin{theorem} 
[H.~Poritsky, J.~M.~Whittaker 1932] \label{TheoremPoristkyUnicite-n=1}
Let $f$ be an entire function of exponential type $<\pi$ satisfying $f^{(t)}(0)=f^{(t)}(1)=0$ for all sufficiently large $t \in 2\N$.
Then $f$ is a polynomial.
\end{theorem}

\begin{proof}
We combine some arguments that we used for polynomials in the three proofs of Lemma  \ref{Lemma:UneVariable}. 
Let $\ftilde=f-P$, where $P$ is the polynomial satisfying 
$$
P^{(t)}(0)=f^{(t)}(0) \and P^{(t)}(1)=f^{(t)}(1) \for t \in 2\N.
$$
We have ${\ftilde}^{(t)}(0)={\ftilde}^{(t)}(1)=0$ for all $t \in 2\N$.
 
The functions $\ftilde(z)$ and $\ftilde(1-z)$ are odd, hence $\ftilde(z)$ is periodic of period $2$. Therefore there exists a function $g$, analytic in $\C^\times$, such that $\ftilde(z)=g(\rme^{\pi\rmi z})$. Since $\ftilde(z)$ has exponential type $<\pi$, using Cauchy's inequalities for the  coefficients of the Laurent expansion of $g$ at the origin, we deduce $\ftilde=0$ and $f=P$. 
\end{proof}

Theorem \ref{TheoremPoristkyUnicite-n=1} is best possible in the following two directions:
\\
$\bullet$ The entire function $\sin(\pi z)$ has exponential type $\pi$ and satisfies $f^{(t)}(0)=f^{(t)}(1)=0$ for all $t \in 2\N$.
\\
$\bullet$ If we 
assume only $f^{(t)}(0)=f^{(t)}(1)=0$ for all  even $t$  outside a finite set, then the conclusion is still valid -- this follows from 
Theorems \ref{Theorem:Lidstone} and \ref{TheoremPoristkyUnicite-n=1}. 
However, if we 
remove an infinite subset of conditions in the assumptions of  Theorem \ref{TheoremPoristkyUnicite-n=1}, then the conclusion is no more valid:
\begin{lemma}\label{Lemma:contrexemple-n=1}
 Let $\calE$ be an infinite subset of the set of $(t,i)\in 2\N\times \{0,1\}$. Then there exists a non countable set of transcendental entire functions $f$ of order $0$, with rational Taylor coefficients at the origin, such that $f^{(t)}(i)=0$ for all $(t,i)\in (2\N\times \{0,1\} )\setminus\calE$. 
\end{lemma}

\begin{proof} 
Let $(P_m)_{m\ge 0}$ be an infinite sequence of polynomials belonging to the set $\{\Lambda_{t,i}\; \mid\; (t,i)\in\calE\}$. Let  $d_m$ be the degree of $P_m$. We assume the sequence $(d_m)_{m\ge 0}$ to be increasing. Let $(c_m)_{m\ge 0}$ be a sequence of rational numbers such that 
$$
|P_m|_r\le |c_m| r^{d_m}
$$
for all $r\ge 1$ and $m\ge 0$. For $m\ge 0$, set
$$
u_m=\frac 1 {c_m (d_m!)^2} \cdotp
$$
The series 
$$
\sum_{m\ge 0} u_mP_m(z)
$$
is uniformly convergent on any compact subset of $\C$, its sum $f(z)$ is an entire function of order $0$. From the uniform convergence of the series,  we deduce,  for all 
$t\in 2\N$ and $i\in\{0,1\}$,  
$$
f^{(t)}(i)=
\begin{cases}
u_m & \hbox {if $P_m=\Lambda_{t,i}$},
\\
0 & \hbox{ if $P_m\not=\Lambda_{t,i}$}.
\end{cases}
$$
The conclusion of Lemma \ref{Lemma:contrexemple-n=1} follows. 
\end{proof}

\section{Expansion of entire functions and generating series}\label{SS:expansion1}

Lidstone finite expansion for polynomials \eqref{Equation:LidstoneExpansionUneVariable} has been extended in  \cite{Poritsky,Whittaker} to an infinite expansion  for entire functions of exponential type $<\pi$ as follows: 
 
 \begin{theorem}
[H.~Poritsky,  J.~M.~Whittaker 1932]\label{Theorem:Poritsky} 
The expansion  \eqref{Equation:LidstoneExpansionUneVariable}  
holds for any entire function $f$ of exponential type $<\pi$, where, for each $z\in\C$, the series 
$$
\sum_{t  \in 2\N} f^{(t)}(0) \Lambda_{t,0}(z)
 \and 
\sum_{t  \in 2\N}f^{(t)}(1) \Lambda_{t,1}(z) 
$$
are absolutely convergent. 
 \end{theorem}
 
 Notice that Theorem \ref{TheoremPoristkyUnicite-n=1} is a consequence of Theorem  \ref{Theorem:Poritsky}. 
  
 We will deduce from Theorem  \ref{Theorem:Poritsky} explicit formulae for the two following generating series:
 $$
 M_1(\zeta,z):=\sum_{t\in 2\N} \Lambda_{t,1}(z) \zeta^t
 \and
 M_0(\zeta,z):=\sum_{t\in 2\N} \Lambda_{t,0}(z) \zeta^t. 
 $$
 
 \begin{corollary}\label{Corollary:generatingseries}
For $|\zeta|<\pi$, we have 
 \begin{equation}\label{Equation:SerieGeneratriceUneVariableM1}
 M_1(\zeta,z) 
=\frac{\sinh(\zeta z)}{\sinh (\zeta )}
\end{equation} 
and 
\begin{equation}\label{Equation:SerieGeneratriceUneVariableM0}
M_0(\zeta,z)
=\cosh(\zeta z) - \sinh(\zeta z)\coth (\zeta ).
\end{equation}
 \end{corollary}
 
Since $\Lambda_{t,0}(z)=\Lambda_{t,1}(1-z)$, we have $M_0(\zeta,z)=M_1(\zeta,1-z)$ and the trigonometric relation
$$
 \sinh(z_1-z_2)=\sinh(z_1)\cosh(z_2)-\cosh(z_1)\sinh(z_2)
$$
shows that the two formulae   \eqref{Equation:SerieGeneratriceUneVariableM1}  and  \eqref{Equation:SerieGeneratriceUneVariableM0}  are equivalent.

 \begin{proof}[Proof of \eqref{Equation:SerieGeneratriceUneVariableM1} as a consequence of Theorem  \ref{Theorem:Poritsky}]
 Let $\zeta\in\C$ satisfy $|\zeta|<\pi$.
We use Theorem  \ref{Theorem:Poritsky}  and formula \eqref{Equation:LidstoneExpansionUneVariable} 
 for the function $f_\zeta(z)=\rme^{\zeta z}$. 
Since $f^{(t)}_\zeta (0)= \zeta^t$ and $f^{(t)}_\zeta (1)= \rme^\zeta \zeta^t$,  we deduce
\begin{equation}\label{Equation:ExpansionExpzetaz}
\rme^{\zeta z}=
\sum_{ t  \in 2\N}
 \Lambda_{t,0}(z)\zeta^t+
\rme^\zeta 
\sum_{ t  \in 2\N}\Lambda_{t,1}(z)\zeta^t.
\end{equation} 
Replacing $\zeta $ with $-\zeta $ yields 
$$
\rme^{-\zeta z}=
\sum_{ t \in 2\N}
  \Lambda_{t,0}(z)\zeta^t+
\rme^{-\zeta } 
\sum_{ t \in 2\N}
  \Lambda_{t,1}(z)\zeta^t.
$$ 
Hence
$$
\rme^{\zeta z}-\rme^{-\zeta z}=(\rme^\zeta -\rme^{-\zeta })
\sum_{ t  \in 2\N}
 \Lambda_{t,1}(z) \zeta^t.
$$  
 This proves \eqref{Equation:SerieGeneratriceUneVariableM1}. 
\end{proof}

From \eqref{Equation:SerieGeneratriceUneVariableM1} one readily deduces the following relation \cite[Equation(3.5)]{Whittaker} for $t\in2\N$,
$$
\Lambda_{t,1}(z)=\frac{2^{t+1}}{(t+1)!} B_{t+1}\left(\frac {1+z} 2\right),
$$
between the Lidstone polynomials and the Bernoulli polynomials; the latter are defined by 
$$
t\frac{\rme^{tz}-1}{\rme^t -1}=
\sum_{n=1}^\infty B_n(z)\frac {t^n}{n!}\cdotp
$$
From Corollary  \ref{Corollary:generatingseries}  we deduce $ M_1(\zeta,z+1)-M_1(\zeta,z-1)= 2\cosh(\zeta z)$, which means 
$$
 \Lambda_{t,1}(z+1)- \Lambda_{t,1}(z-1)= 2\frac {z^t}{t!}\cdotp
 $$ 
This relation also follows from the functional equation 
$$
B_n(z+1)- B_n(z)= nz^{n-1}
$$ 
of the Bernoulli polynomials.

Our proof of Theorem  \ref{Theorem:Poritsky}  below will rest on \eqref{Equation:SerieGeneratriceUneVariableM1}, hence we need to give a direct proof of it. 
 \begin{proof}[Direct proof of  \eqref{Equation:SerieGeneratriceUneVariableM1}]
 We start with the formula
 \begin{equation}\label{Equation:expzetaz}
\rme^{\zeta z}=\frac{\sinh(\zeta (1-z) )}{\sinh (\zeta )}+\rme^\zeta  \ \frac{\sinh(\zeta z)}{\sinh (\zeta )},
\end{equation}
which holds for $\zeta \in\C$, $\zeta \not\in \pi\rmi \Z$ and $z\in\C$. For $\zeta \in\C$, $\zeta \not\in \pi\rmi \Z$,  the entire function 
$$
f(z)=\frac{\sinh(\zeta z)}{\sinh (\zeta )}=\frac{\rme^{\zeta z}-\rme^{-\zeta z}}{\rme^\zeta -\rme^{-\zeta }}
$$
satisfies 
$$
f''=\zeta ^2f,\quad f(0)=0,\quad f(1)=1,
$$
hence $f^{(t)}(0)=0$ and $f^{(t)}(1)=\zeta ^t$ for all $t \in 2\N$.

  For $z\in\C$ and $|\zeta |<\pi$, let 
$$
F(\zeta,z)=
\frac{\sinh(\zeta z)}{\sinh (\zeta )} 
$$
with $F(0,z)=z$.  
Fix $z\in\C$. The map $\zeta \mapsto F(\zeta,z )$ is analytic in the disc $|\zeta |<\pi$ and is  even: $F(-\zeta,z )=F(\zeta,z )$. Consider its Taylor expansion at the origin: 
$$
F(\zeta,z )=
\sum_{ t  \in 2\N}
c_t(z) \zeta^t
$$
with $c_0(z)=z$.  For fixed $z\in\C$, this Taylor series is absolutely and uniformly convergent on any compact subset of the disc $|\zeta |<\pi$.
We have $F(\zeta,0 )=0$,  $F(\zeta,1 )=1$, and
 $$
F(\zeta,z )=\frac{\rme^{\zeta z}-\rme^{-\zeta z}}{\rme^\zeta -\rme^{-\zeta }}\cdotp
$$
From 
$$
c_t(z)=
\frac{1}{t!} 
\left(\frac{\partial}{\partial \zeta}\right)^{t}F(0,z)
$$
it follows that $c_t(z)$ is a polynomial.  
From 
$$
\left(\frac{\partial}{\partial z}\right)^2F(\zeta,z )=\zeta ^2F(\zeta,z )
$$
we deduce 
$$
c''_t=c_{t-2} 
\for
 t\in 2\N, \;  t\ge 2.
 $$
 Since $c_t(0)=c_t(1)=0$ for $t\in 2\N$, $t\ge 2$, we deduce from Lemma \ref{Lemma:equadiffunevariable} that $c_t(z)=\Lambda_{t,1}(z)$.

This completes the proof of  
  \eqref{Equation:SerieGeneratriceUneVariableM1}, hence the proof of  \eqref{Equation:ExpansionExpzetaz} for all $\zeta\in\C$ with $|\zeta|<\pi$.
   \end{proof}
   
We are going to prove Theorem \ref{Theorem:Poritsky} by means of the Laplace transform, which is a special case of the method of kernel expansion of R.C. Buck   \cite{Buck}; see also   \cite[Chap.I \S 3]{BoasBuck}.
 Let 
$$
f(z)=\sum_{k\ge 0} \frac{a_k}{k!} z^k
$$ 
be an entire function of exponential type $\tau(f)$. 
The \emph{Laplace transform} of $f$, viz.
\begin{equation}\label{Equation:Laplace}
F(\zeta )=\sum_{k\ge 0} a_k\zeta ^{-k-1}, 
\end{equation}
is analytic in the domain $\{\zeta\in\C\; \mid \; |\zeta |>\tau(f)\}$. 
From Cauchy's residue Theorem we deduce, for $r>0$,
\begin{equation}\label{Equation:Cauchy}
\frac{1}{2\pi\rmi} \int_{|\zeta|=r} \rme^{\zeta  z}\zeta^{-k-1}\rmd \zeta =\frac {z^k}{k!}\cdotp
\end{equation}
From the absolute and uniform convergence of the series in the right hand side of  \eqref{Equation:Laplace} on $|\zeta|=r$, it follows that for $r>\tau(f)$ we have
$$
f(z) =\frac{1}{2\pi\rmi} \int_{|\zeta|=r} \rme^{\zeta  z}F(\zeta ) \rmd \zeta 
$$  
and
$$
f^{(t)}(z) =\frac{1}{2\pi\rmi} \int_{|\zeta |=r} \zeta^t\rme^{\zeta  z}F(\zeta ) \rmd \zeta .
$$
   
  \begin{proof}[Proof of Theorem \ref{Theorem:Poritsky}]
Let $f$ be an entire function of exponential type $\tau(f)$ satisfying  $\tau(f)<\pi$. Let $r$ satisfy $ \tau(f)<r<\pi$. 
From the uniform convergence of the series  \eqref{Equation:ExpansionExpzetaz} on the compact set $\{\zeta\in\C\; \mid \; |\zeta|=r\}$, we deduce
$$
\begin{aligned}
f(z) = \sum_{t  \in 2\N}
\left( \frac{1}{2\pi\rmi} \int_{|\zeta |=r} \zeta^t F(\zeta ) \rmd \zeta \right)&
\Lambda_{t,0}(z)+
\\
&
\sum_{t \in 2\N}
\left( \frac{1}{2\pi\rmi} \int_{|\zeta |=r} \zeta^t \rme^\zeta  F(\zeta ) \rmd \zeta \right)\Lambda_{t,1}(z),
\end{aligned}
$$
and therefore (formula of Poritsky and Whittaker \eqref{Equation:LidstoneExpansionUneVariable} for entire functions of exponential type $<\pi$) 
$$
f(z) = 
\sum_{t \in 2\N}
 f^{(t)}(0) \Lambda_{t,0}(z)+
 \sum_{t \in 2\N} f^{(t)}(1) \Lambda_{t,1}(z),
$$
where the two series are absolutely convergent.  

This completes the proof of Theorem \ref{Theorem:Poritsky}.
\end{proof}
  
 \section{Integral formulae for Lidstone polynomials}\label{SS:IntegralFormulaLidstone1}
 
 Using Cauchy's residue Theorem, we deduce  from \eqref{Equation:SerieGeneratriceUneVariableM1} the following integral formula \cite[(4.1)]{Whittaker}:
 
 \begin{proposition}\label{Prop:IntegralFormulaLidstone}
 For $z\in\C$, $t \in 2\N$ and   $K\ge 0$, we have
$$ 
\begin{aligned}
\Lambda_{t,1}(z)
=(-1)^{t/2}\frac{2}{\pi^{t+1} } 
&
\sum_{k=1}^{K} \frac{(-1)^{k+1}}{k^{t+1}} 
\sin \bigl( k\pi z\bigr)
\\ 
&
\qquad
+\frac{1}{2\pi\rmi} \int_{|\zeta|=(2K+1)\pi/2} \zeta^{-t-1}\frac
{\sinh(\zeta z)}{\sinh(\zeta)}
\rmd \zeta.
\end{aligned} 
$$
\end{proposition}

\begin{proof}
Let $z\in\C$. Inside the disc 
$\{\zeta\in\C\; \mid \; |\zeta|\le (2K+1)\pi/2\}$, the function $\zeta\mapsto \zeta^{-t-1}\frac
{\sinh(\zeta z)}{\sinh(\zeta)}
$ 
has a pole of order $t+1$ at $\zeta=0$ and only simple poles at $\zeta=k\pi\rmi$ with $k\in\Z$, $0<|k|\le K$.
The residue at $0$ is $\Lambda_{t,1}(z)$, while for  $k\in\Z\setminus\{0\}$, the residue at $k\pi\rmi$ is
\begin{equation}\label{equation:residu}
(-1)^k 
\rmi^{-t} (k\pi)^{-t-1} 
\sin \bigl( k\pi z\bigr).
\end{equation}
Since $t$ is even, the function is odd and the residues at $k\pi\rmi$ and at $-k\pi\rmi$ are the same. 
\end{proof}

In particular, with $K=1$ we have \cite[(4.3)]{Whittaker}
$$
\Lambda_{t,1}(z)=(-1)^{t/2}\frac{2}{\pi^{t+1} } \sin (\pi z)
+
\frac{1}{2\pi\rmi} \int_{|\zeta|=3\pi/2} \zeta^{-t-1}\frac
{\sinh(\zeta z)}{\sinh(\zeta)}
\rmd \zeta.
$$ 
Since $|\sinh(\zeta)|\ge 1$ for $|\zeta|=3\pi/2$, 
one deduces,
for $t\ge 0$ and $r>0$,  
\begin{equation}
\label{Equation:MajorationLambdat1}
\left\lbrace
\begin{aligned}
\left|\Lambda_{t,1}(z) - (-1)^{t/2} \frac 2 {\pi^{t+1}}\sin (\pi z) \right| 
&\le  \left(\frac 2 {3\pi} \right)^t 
\rme^{3\pi r/2},
\\
\left|\Lambda_{t,0}(z) - (-1)^{t/2} \frac 2 {\pi^{t+1}}\sin (\pi z) \right| 
&\le \rme^{3\pi/2} \left(\frac 2 {3\pi} \right)^t 
\rme^{3\pi r/2}.
\end{aligned}
\right.
\end{equation} 
These estimates 
enable Whittaker \cite[Theorem 1]{Whittaker} to solve the Lidstone interpolation problem as follows.
 Let $(a_t)_{t \in 2\N}$ and $(b_t)_{  t \in 2\N}$ be two sequences of complex numbers. If the series
$$
\sum_{t\in2\N}(-1)^{t/2}\frac{a_t}{\pi^t}
\and
\sum_{t\in2\N}(-1)^{t/2}\frac{b_t}{\pi^t}
$$ 
are convergent, then  
\begin{equation}\label{Equation:LidstoneInterpolationSolution}
\sum_{t\in2\N} a_t \Lambda_{t,0}(z) + \sum_{t\in2\N} b_t \Lambda_{t,1}(z)
\end{equation}
is uniformly convergent on any compact of $\C$ and its sum $f(z)$
is an  
entire function  
satisfying 
$$
 f^{(t)}(0)=a_t
\and
 f^{(t)}(1)=b_t
 \hbox{ for all $  t \in 2\N$}.
 $$
  If one of  the series
$$
\sum_{t\in2\N}(-1)^{t/2}\frac{a_t}{\pi^t},
\quad
\sum_{t\in2\N}(-1)^{t/2}\frac{b_t}{\pi^t}
$$ 
is not convergent, then \eqref{Equation:LidstoneInterpolationSolution} cannot converge for any non integral value of $z$. 

Another consequence of \eqref{Equation:MajorationLambdat1} is, for $t\in 2\N$ and $r\ge 0$,
\begin{equation}\label{Equation:BoundLambda}
|\Lambda_{t,1}|_r
\le 2 \pi^{-t} \rme^{3\pi r/2}
\and
|\Lambda_{t,0}|_r
\le 2  \rme^{3\pi /2}\pi^{-t} \rme^{3\pi r/2}.
\end{equation}
 
 We now prove another integral formula for the polynomials $\Lambda_{t,0}$.

\begin{proposition}
\label{Prop:IntegralFormula2}
For $t\in 2\N$   and for $K\ge 0$, we have
$$ 
\begin{aligned}
\Lambda_{t,0}(z)=\frac{z^t}{t!}+&(-1)^{t/2}\frac{2}{\pi^{t+1} } 
\sum_{k=1}^{K} \frac{1}{k^{t+1}} 
\sin \bigl( k\pi z\bigr)
\\ 
&
\qquad
-\frac{1}{2\pi\rmi} \int_{|\zeta|=(2K+1)\pi/2} \zeta^{-t-1} 
\sinh(\zeta z) \coth(\zeta)
\rmd \zeta.
\end{aligned} 
$$ 
\end{proposition}

\begin{proof}
Proposition \ref{Prop:IntegralFormula2} is equivalent to Proposition \ref{Prop:IntegralFormulaLidstone} by changing the variable $z$ to $1-z$. We give another proof by repeating the same arguments as for the proof of Proposition \ref{Prop:IntegralFormulaLidstone}.
Inside the disc 
$\{\zeta\in\C\; \mid \; |\zeta|\le (2K+1)\pi/2\}$,  the function 
$
\zeta\mapsto
\zeta^{-t-1} 
\sinh(\zeta z) \coth(\zeta)$
has only simple poles at $k\pi\rmi$ with $k\in\Z$, $|k|\le K$.
From
\eqref{Equation:SerieGeneratriceUneVariableM0},
it follows that the residue at $\zeta=0$ is 
$$
\frac{z^t}{t!} -\Lambda_{t,0}(z),
$$ 
while for  $k\in\Z\setminus\{0\}$, the residue at $k\pi\rmi$ is
\begin{equation}\label{equation:residupsi}
\rmi^{-t} (k\pi)^{-t-1} 
\sin \bigl( k\pi z\bigr).
\end{equation}
Since $t$ is even, the function we integrate is odd and the residues at $k\pi\rmi$ and at $-k\pi\rmi$ are the same. 
 \end{proof}

 \section{Functions of finite exponential type}\label{SS:finiteExponentialType-n=1}
 
 We follow \cite{Buck}. Let $K\ge 1$. The function $\zeta\mapsto  \frac {\sinh (\zeta z)}{\sinh(\zeta)}$ is even and has only  simple poles at $k\pi\rmi$ with $k\in\Z\setminus\{0\}$,  
with residue given by \eqref{equation:residu} with $t=-1$, namely  $(-1)^k \rmi \sin(k\pi z)$. The sum of the residues at  $k$ and $-k$ is $0$ and we have\footnote{A factor $2$ is missing in \cite[p.795]{Buck} and \cite[Chap.I \S 4 p.15]{BoasBuck}.}
 $$
 \frac{(-1)^k \rmi \sin(k\pi z)}{\zeta-k\pi \rmi} -  \frac{(-1)^k \rmi \sin(k\pi z)}{\zeta+k\pi \rmi}
 =2\pi(-1)^{k+1}\frac {k\sin(k\pi z)}{\zeta^2+k^2\pi^2}\cdotp
 $$ 
 Hence the function $ G_K(\zeta,z)$ defined by 
 \begin{equation}\label{Equation:GK}
 \frac {\sinh (\zeta z)}{\sinh(\zeta)}= 
 2\pi\sum_{k=1}^K \frac{(-1)^{k+1}k\sin(k\pi z)}{\zeta^2+k^2\pi^2}
 +G_K(\zeta,z)
\end{equation}
 is analytic in the domain  $\bigl\{(\zeta,z)\in\C^2\; \mid \; |\zeta|<(K+1)\pi\bigr\}$.  Notice that for $|\zeta|<k\pi$ and $k\ge 1$ we have
 $$
 \frac 1{\zeta^2+k^2\pi^2}=\sum_{t\in 2\N} \frac{(\rmi\zeta)^t}{(k\pi)^{t+2}}\cdotp
 $$
 The function $z\mapsto G_K(\zeta,z)$ is odd. 
Since the function $\zeta\mapsto G_K(\zeta,z)$ is even, its Taylor expansion at the origin 
can be written
  $$
 G_K(\zeta,z)=\sum_{t \in 2\N}g_t(z) \zeta^t
 $$
 where the functions $g_t(z)$ are odd entire functions. 
This Taylor series is absolutely and uniformly convergent for $\zeta$ in any compact subset of the disc $\{\zeta\in\C\; \mid \; |\zeta|<(K+1)\pi\}$. The Taylor coefficient  $g_t(z) $ is the sum of $\Lambda_{t,1}(z)$ and a finite  trigonometric sum of exponential type $\le K \pi$, namely
$$
g_t(z)=\Lambda_{t,1}(z) 
+2(-1)^{t/2}\sum_{k=1}^K  (-1)^k (k\pi)^{-t-1}\sin(k\pi z). 
$$
 Using \eqref{Equation:expzetaz} we deduce, for $|\zeta|<(K+1)\pi$,  
\begin{equation}\label{Equation:expzetazAvecgt}
 \rme^{\zeta z}=\sum_{t \in 2\N}g_t(1-z) \zeta^t  + \rme^\zeta \sum_{ t\in 2\N}g_t(z) \zeta^t 
 +
2\pi\sum_{k=1}^K \frac{ k\sin(k\pi z)}{\zeta^2+k^2\pi^2}\left(1+(-1)^{k+1} \rme^\zeta\right).
\end{equation} 
 
 \begin{proposition}[R.C.~Buck, 1955]\label{Prop:Buck}
Let $K$ be a positive integer.  Let $f$ be an entire function of finite exponential type $\tau(f)<(K+1)\pi$ and let $F(\zeta)$ be the Laplace transform of $f$.  Then for $z\in\C$ we have 
$$
f(z)=\sum_{t \in 2\N}f^{(t)}(0) g_t(1-z)+\sum_{t \in 2\N } f^{(t)}(1)g_t(z)+
\sum_{k=1}^K C_k \sin(k\pi z),
$$
where the series are absolutely convergent and 
\begin{equation}\label{Equation:Ck}
C_k=
{-k\rmi}  \int_{|\zeta|=r}
\frac{1+(-1)^{k+1} \rme^\zeta}
{\zeta^2+k^2\pi^2}
F(\zeta ) \rmd \zeta \quad (1\le k\le K)
\end{equation}
 for any $r$ in the range $\tau(f)<r<(K+1)\pi$.
\end{proposition}

\begin{proof}
Let $r$ satisfy $\tau(f)<r<(K+1)\pi$. 
From the absolute and uniform convergence on $|\zeta|=r$ of the series in the right hand side of \eqref{Equation:expzetazAvecgt}, we deduce 
$$
\begin{aligned}
f(z) 
&=\frac{1}{2\pi\rmi} \int_{|\zeta|=r} \rme^{\zeta  z}F(\zeta ) \rmd \zeta 
\\
&=
\sum_{ t \in 2\N}g_t(1-z) 
\frac{1}{2\pi\rmi} \int_{|\zeta|=r} \zeta^t F(\zeta ) \rmd \zeta 
 +\sum_{t \in 2\N}g_t(z)  
\frac{1}{2\pi\rmi} \int_{|\zeta|=r} \zeta^t \rme^\zeta F(\zeta ) \rmd \zeta \\
&\qquad\qquad\qquad\qquad\qquad\qquad\qquad\qquad
+
\sum_{k=1}^K C_k \sin(k\pi z),
\end{aligned}
$$
with
$$
\frac{1}{2\pi\rmi} \int_{|\zeta|=r} \zeta^t F(\zeta ) \rmd \zeta =f^{(t)}(0)
 \and
 \frac{1}{2\pi\rmi} \int_{|\zeta|=r} \zeta^t \rme^\zeta F(\zeta ) \rmd \zeta=f^{(t)}(1).
  $$
\end{proof}

\noindent
{\bf Example.} The Laplace transform of $f(z)=\sin(\pi z)$ is $F(\zeta)=\frac \pi {\zeta^2+\pi^2}$ and for $\pi<r<2\pi$ we have
$$
\int_{|\zeta|=r}\frac {1+\rme^\zeta}{(\zeta^2+\pi^2)^2}\rmd \zeta= \frac  \rmi \pi,
$$
hence for this function $f$, we have
$$
C_1=
-\rmi  \int_{|\zeta|=r}
\frac{1+\rme^\zeta}
{\zeta^2+\pi^2}
F(\zeta ) \rmd \zeta 
=
-\rmi  \pi \int_{|\zeta|=r}
\frac{1+\rme^\zeta}
{(\zeta^2+\pi^2)^2}
 \rmd \zeta 
=1,
$$
as expected.

In \cite{BivariateLidstone} and \cite{MultivariateLidstone}, we will need the following variant of \eqref{Equation:GK}. 
The  function  $ H_K(\zeta,z)$   defined by
\begin{equation}\label{Equation:HK} 
 \sinh(\zeta z) \coth(\zeta)
= -2\pi\sum_{k=1}^K \frac{ k\sin(k\pi z)}{\zeta^2+k^2\pi^2}
 +H_K(\zeta,z),
\end{equation}
 is analytic in the domain $\bigl\{(\zeta,z)\in\C^2\;\mid\; 
 |\zeta|<(K+1)\pi \bigr\}$. The map $\zeta\mapsto H_K(\zeta,z)$ is even and the map $z\mapsto H_K(\zeta,z)$ is odd. 
Replacing $z$ with $1-z$ in \eqref{Equation:GK} yields
$$
H_K(\zeta,z)=\cosh(\zeta z)-G_K(\zeta,1-z).
$$

 \begin{corollary}[I.J.~Schoenberg, 1936]
 \label{Corollary:Schoenberg}
Let $f$ be an entire function of finite exponential type $\tau(f)$ satisfying $f^{(t)}(0)=f^{(t)}(1)=0$ for all $t \in 2\N$. Then 
$$ 
f(z)= \sum_{k=1}^K C_k\sin(k\pi z).
$$ 
with $K\le \tau(f)/\pi$ and with  the constants $C_1,\dots,C_K$  given by \eqref{Equation:Ck}.
\end{corollary}

A side result is that the exponential type $\tau(f)$ of a function $f$ satisfying the assumption of Corollary \ref{Corollary:Schoenberg} is an integer multiple of $\pi$.

\end{document}